\documentclass[12pt]{amsart}

\usepackage{melanie}

\usepackage{fullpage}

\newtheorem{nts}{Note to self}

\newcommand{\N}{\mathbb{N}}
\newcommand{\cF}{\mathcal{F}}

\newcommand{\s}{{\bf s}}

\title{The distribution of points on superelliptic curves over finite fields}

\author{GilYoung Cheong, Melanie Matchett Wood, and Azeem Zaman}
\address{Department of Mathematics\\
University of Wisconsin-Madison \\ 480 Lincoln Drive \\
Madison, WI 53705 USA\\
and
American Institute of Mathematics\\360 Portage Ave \\
Palo Alto, CA 94306-2244 USA} 
\email{mmwood@math.wisc.edu}

\begin{document}

\begin{abstract}
We give the distribution of points on smooth superelliptic curves over a fixed finite field, as their degree goes to infinity.
We also give the distribution of points on smooth $m$-fold cyclic covers of the line, for any $m$, as the
degree of their superelliptic model goes to infinity.  This builds on previous work of Kurlberg, Rudnick, Bucur, David, Feigon, and Lal{\'{\i}}n for $p$-fold cyclic covers, but the limits taken differ slightly and the resulting distributions are interestingly different.
\end{abstract}
\maketitle

\section{Introduction}

For a family of smooth curves over a fixed finite field $\F_q$, it is natural to ask how many $\F_q$ points those curves have on average, or more precisely, what distribution of points one obtains from a random curve in the family.  
This question has been studied as the genus, degree, or related invariants go to infinity in many cases including: hyperelliptic curves \cite{kurlberg-rudnick}, cyclic trigonal curves \cite{bdfl-trigonal}, cyclic $p$-fold curves \cite{bdfl-pfold}, plane
curves \cite{bdfl-planar}, complete intersections in projective space \cite{bucur-kedlaya}, trigonal curves \cite{wood-trigonal},
and curves in Hirzebruch surfaces \cite{erman-wood-semiample}.

In this paper, we give the distribution of points on smooth superelliptic (affine) curves, $C_f$, given by equations
$$
y^m=f(x)
$$
for some $f(x)\in\F_q[x]$, as the degree of $f$ goes to infinity.  We also give the distribution for the normalizations of
(possibly singular) super-elliptic curves, which for $q\equiv 1 \pmod{m}$ is exactly the case of $m$-fold cyclic covers of the line (c.f. with the work on prime degree cyclic covers cited above).  
One important advance in our work is that the normalizations may have different numbers of points from their singular models,
and it is general difficult to write down the smooth curves explicitly.  The above cited works have all counted points on curves
with explicitly given equations.  In Section~\ref{S:ptsonnorm} of this paper, we relate the number of points on a smooth curve to its explicitly given superelliptic model.  

In contrast to studying a similar question when a family of fixed genus is chosen and $q\ra\infty$, when the philosophy
of Katz-Sarnak \cite{katz-sarnak-a} suggests that the limit distributions should be predicted by a certain group of random matrices;
in the large genus limit for fixed $q$, there is no general conjectural picture of what one should expect.  Thus it is important to have many examples of different families exhibiting different phenomena.  
In the cases we study,  we obtain a particularly interesting contrast to \cite{bdfl-trigonal,bdfl-pfold}, as we count the same cyclic
covers of the line, but with a different invariant going to infinity, and obtain different distributions.  

Throughout, $m$ will be fixed and $(q,m)=1$.
We write $\F_q[x]_d$ to denote the degree $d$ polynomials in $\F_q[x]$.
Note that $C_f$ is smooth if and only if $f(x)$ is square-free, for which curves we have the following result.
\begin{theorem}\label{T:sqfree}
Let  $(q,m)=1$.  Then
 $$
\lim_{d\ra\infty} \frac{\#\{f\in \F_q[x]_d | {C_f} \textrm{ smooth with $k$ $\F_q$ points} \}}{\# \{f\in \F_q[x]_d | {C_f} \textrm{ smooth} \}}=
\Prob(\sum_{i=1}^q X_i=k ),
$$
where the $X_{j}$ are independent and identically distributed random variables with
$$
X_{j}
	=
	\begin{cases}
		0 \text{ with probability }	\left(1-\frac{1}{(m,q-1)} \right)\frac{1}{1+q^{-1}}, \\
		1 \text{ with probability }	\frac{q^{-1}}{1+q^{-1}}, \\
		(m,q-1) \text{ with probability } \frac{1}{(m,q-1)} \frac{1}{1+q^{-1}}.
	\end{cases},
$$
where when $(m,q-1)=1$ the last two probabilities are added to give $\Prob(X_j=1)$.
We have the same results if we restrict to $f$ such that $C_f$ is irreducible or geometrically irreducible.
\end{theorem}
In the case $m=2$, Theorem~\ref{T:sqfree} reduces to \cite[Theorem 1]{kurlberg-rudnick} giving the distribution of $\F_q$ points
on hyperelliptic curves.

When $f(x)$ is not square-free, we can consider the normalization $\tilde{C_f}$ of $C_f$, which is a smooth curve.  
Note that $C_f$ and $C_{fg^m}$ have the same normalization, for $g\in\F_q[x]$.  So it is natural to consider the smooth normalizations $\tilde{C_f}$ for $f$ that are $m$th power-free.
\begin{theorem}\label{T:mfree}
Let  $(q,m)=1$.  Then
 $$
\lim_{d\ra\infty} \frac{\#\{f\in \F_q[x]_d | \textrm{$f$ $m$th power-free, } {\tilde{C_f}} \textrm{ with $k$ $\F_q$ points} \}}{\# \{f\in \F_q[x]_d | \textrm{$f$ $m$th power-free}\}}=
\Prob(\sum_{i=1}^q X_i=k),
$$
where the $X_{j}$ are independent and identically distributed random variables with
$$
X_{j}
	=
	\begin{cases}
		0 \text{ with probability }	\displaystyle\sum_{0 \leq s \leq m-1}\left(1-\dfrac{1}{(m, s, q-1)}\right)\frac{1-q^{-1}}{1-q^{-m}}, \\
		N \text{ with probability } \displaystyle\sum_{\substack{0 \leq s \leq m-1 \\ (m, s,q-1) = N}}\dfrac{q^{-s}(1-q^{-1})}{N(1-q^{-m})}.
	\end{cases}
$$
We have the same results if we restrict to $f$ such that  $\tilde{C_f}$ is irreducible or geometrically irreducible.
\end{theorem}
If we are interested in $m$-fold cyclic covers of the line, then we must consider $q\equiv 1 \pmod{m}$.
When $q\equiv 1 \pmod{m}$, by Kummer theory, the cyclic $m$-fold covers of the line are exactly the irreducible
$\tilde{C_f}$ of Theorem~\ref{T:mfree}.
Both Theorems~\ref{T:sqfree} and \ref{T:mfree} will follow from Theorem~\ref{T:nfree}, which will give the distribution of points on superelliptic curves and their normalizations with $f(x)$ that are $n$th power-free,
further refined by $x$ values.  We prove Theorem~\ref{T:nfree} by first relating the number of points
on the normalizations to the explicit affine model in Section~\ref{S:ptsonnorm}, and then using counting methods
similar to those of \cite{kurlberg-rudnick} in Section~\ref{S:main} (with the difference that our setup allows us to avoid counting
polynomials interpolated to take zero values).

Note that the average number of points in both theorems above is $q$, which agrees
with the average in \cite{bdfl-trigonal,bdfl-pfold,kurlberg-rudnick} for prime degree cyclic covers of the (affine) line.
However, our distributions differ from those in \cite{bdfl-trigonal,bdfl-pfold}, which is not a contradiction as we take
different invariants going to infinity.  In our case, we are letting the degree of $f$, or $C_f$, the affine model of the curve,
go to infinity.  Cyclic $p$-fold curves have a signature $d_1,\dots, d_{p-1}$, and in \cite{bdfl-trigonal,bdfl-pfold}
one has $\min_i(d_i)\ra \infty.$  It would be very interesting to understand why and how this difference has an impact on the resulting distribution of points, not only in this case, but more generally.  For this reason, we particularly highlight the first 
contrast, in the case of cyclic trigonal curves, in which we have Theorem~\ref{T:mfree} when $q\equiv 1 \pmod{3}$ with
$$
X_{j}
	=
	\begin{cases}
		0 \text{ with probability }	\displaystyle \frac{2}{3(1+q^{-1}+q^{-2})}, \\
		1 \text{ with probability }	\displaystyle \frac{q^{-1}+q^{-2}}{1+q^{-1}+q^{-2}}, and \\
			3 \text{ with probability }	\displaystyle \frac{1}{3(1+q^{-1}+q^{-2})}, \\
	\end{cases}
$$
and
when instead of $\deg(f)\ra\infty$, we have $\min_i(d_i)\ra \infty$, \cite{bdfl-trigonal} gives an analogous theorem with
$$
X_{j}
	=
	\begin{cases}
		0 \text{ with probability }	\displaystyle \frac{2}{3(1+2q^{-1})}, \\
		1 \text{ with probability }	\displaystyle \frac{2q^{-1}}{1+2q^{-1}}, and \\
			3 \text{ with probability }	\displaystyle \frac{1}{3(1+2q^{-1})}. \\
	\end{cases}
$$
The difference seems very suggestive but the reason for it is not immediately clear.

\section{Points on the normalization}\label{S:ptsonnorm}
In this section, we determine the number of points on the normalization $\tilde{C_f}$ in terms of $f$.
Note that we have a map $\tilde{C_f}\ra C_f \ra \A^1=\Spec \F_q[x]$, and so a degree $1$ point of $C_f$ maps to a degree $1$ point of $\A^1=\Spec \F_q[x]$, and thus we can talk about the $x$-value of such a point.  

We say $f\in\F_q[x]$ has \emph{power $r$ over a field $F$} if $r$ is the greatest integer such that $f(x)=g(x)^r$ for
some $g\in F[x]$.
The following lemma deals with when our curve $C_f$ is irreducible or geometrically irreducible.
\begin{lemma}\label{L:irred}
If the power of $f(x)$ over $\bar{\F_q}$ is relatively prime to $m$, then $y^m-f(x)$
is irreducible in $\bar{\F_q}[x,y]$ (i.e. $C_f$ is geometrically irreducible).

\end{lemma}
\begin{proof}
If the power of $f$ over $\bar{\F_q}$ is relatively prime to $m$, then $f$ has order $m$ in $\bar{\F_q}(x)^*/(\bar{\F_q}(x)^*)^m$.
Then, by Kummer theory, we have that $y^m-f(x)$ is irreducible in $\bar{\F_q}[x,y]$.

\end{proof}

\begin{lemma}\label{L:howmany}
Let  $f\in \F_q[x]$ have power over $\bar{\F_q}$ relatively prime to $m$ and write
$f(x)=(x-x_i)^{s_i}g_i(x)$ with $s_i\geq 0$, and $g_i\in\F_q[x]$, and $g_i(x_i)\ne 0$.
Then, the number of 
degree $1$ points of $\tilde{C_f}$ with $x=x_0$ for some $x_0\in\F_q$ is equal to
$$
\begin{cases}
 0 &\textrm{if $g_i(x_i)$ is not an $(m,s_i)$ power in $\F_q$} \\
 (m,s_i, q-1) &\textrm{if $g_i(x_i)$ is an $(m,s_i)$ power in $\F_q$.} 
\end{cases}
$$
\end{lemma}

\begin{proof}
 The number of 
degree $1$ points of $\tilde{C_f}$ with $x=x_0$ for some $x_0\in\F_q$ is equal to the number of degree $1$ primes
over $(x-x_0)$ in $K=\F_q(x)[y]/(y^m-f(x))$ 
by the usual correspondence between curves and their function fields.
(We have that $K$ is a field by Lemma~\ref{L:irred}.)
Without loss of generality, we can consider the case $x_i=0$.  We write $f(x)=x^s g(x)$, where
$g(x)\in\F_q[x]$ and $g(0)\ne 0$.  

We have $z=y^{m/(m,s)}x^{-s/(m,s)}\in K$ and $z^{(m,s)}=g(x)$, so $z$ is integral over $\F_q[x]$.  
Thus at any prime $\wp$ over $(x)$, modulo $\wp$ we have $z^{(m,s)}\equiv g(0)$.  Thus if
$g(0)$ is not an $(m,s)$ power in $\F_q$, then the inertia degree at $\wp$ is $>1$, and there 
are no degree $1$ primes of $K$ over $(x)$.

Note that $K$ contains $L=\F_q(x)[z]/(z^{(m,s)}-g(x))$ (since the power of $f$ over $\bar{\F_q}$ is relatively prime to $m$, the power of $g$ over $\bar{\F_q}$ is relatively prime to $(m,s)$ and thus this polynomial is irreducible by Lemma~\ref{L:irred}).
We have that $L$ is unramified over $(x)$.
Since $\F_q[x,z]/(z^{(m,s)}-g(x))$ is maximal at $(x)$, we can compute the splitting of $(x)$ in 
$L$ by computing the splitting of $z^{(m,s)}-g(x)$ modulo $(x)$.  S
If $g(0)$ is an $(m,s)$ power in $\F_q$, then $z^{(m,s)}-g(x)$ has exactly $(m,s,q-1)$ distinct degree $1$ factors
modulo $(x)$, and there are exactly $(m,s,q-1)$ degree $1$ primes in $L$ over $(x)$. 

We write $(m,s)=mi+sj$ and then $(y^jx^i)^m=x^{mi+sj}g(x)^j$ has valuation $(m,s)$
in $\F_q(x)$ (with respect to $x$), so at any place of $K$ above $(x)$ there is an element of valuation $\frac{(m,s)}{m}$, and thus any prime over $(x)$ has ramification degree $e\geq \frac{m}{(m,s)}$. 
This means every prime of $L$ over $(x)$ must be completely ramified from $L$ to $K$, and thus
there are exactly $(m,s,q-1)$ degree $1$ primes in $K$ over $(x)$.
The lemma follows.
\end{proof}

\section{Main Theorem}\label{S:main}

Throughout this section, we fix an integer $n\geq 2$.
Let $$\cF:=\{f\in \F_q[x] |  \textrm{$f$ is $n$th power-free}\}$$
and $\cF_d$ be the elements of $\cF$ of degree $d$. We write $x_1,\dots,x_q$ for the elements of $\F_q$.

\begin{theorem}\label{T:nfree}
Let $k_1,\dots,k_q$ be integers and $(q,m)=1$.  Then
$$
\lim_{d\ra\infty} \frac{\#\{f\in \cF_d | {C_f} \textrm{ has $k_i$ points with $x=x_i$ for all $i$} \}}{\# {\cF}_d}=
\Prob(X_i=k_i \textrm{ for all $i$}),
$$
where the $X_{j}$ are independent and identically distributed random variables with
$$
X_{j}
	=
	\begin{cases}
		0 \text{ with probability }	\left(1-\frac{1}{(m,q-1)} \right)\frac{1-q^{-1}}{1-q^{-n}}, \\
		1 \text{ with probability }	\frac{q^{-1}-q^{-n}}{1-q^{-n}}, \\
		(m,q-1) \text{ with probability } \frac{1}{(m,q-1)} \frac{1-q^{-1}}{1-q^{-n}}.
	\end{cases},
$$
where when $(m,q-1)=1$ the last two probabilities are added to give $\Prob(X_j=1)$.
Also,
$$
\lim_{d\ra\infty} \frac{\#\{f\in \cF_d | \tilde{C_f} \textrm{ has $k_i$ points with $x=x_i$ for all $i$} \}}{\# {\cF}_d}=
\Prob(X_i=k_i \textrm{ for all $i$}),
$$
where the $X_{j}$ are independent and identically distributed random variables with
$$
X_{j}
	=
	\begin{cases}
		0 \text{ with probability }	\displaystyle\sum_{0 \leq s \leq n-1}\left(1-\dfrac{1}{(m, s, q-1)}\right)\frac{1-q^{-1}}{1-q^{-n}}, \\
		N \text{ with probability } \displaystyle\sum_{\substack{0 \leq s \leq n-1 \\ (m, s,q-1) = N}}\dfrac{q^{-s}(1-q^{-1})}{N(1-q^{-n})}.
	\end{cases}
$$
We have the same results if we restrict to $f$ such that $C_f$ (or $\tilde{C_f}$) is irreducible or geometrically irreducible.
\end{theorem}

\subsection{Notation}
For a prime $h\in\F_q[x]$, we  write $h^s || f$ if $h^s|f$ and $h^{s+1}\nmid f$.
For $s=(s_1,\dots,s_q)\in \N^{q}$, let 
$$\cF^{\s}:=\{f\in \cF | \textrm{ $(x-x_i)^{s_i}|| f$ for all $i$}\},$$
and $\cF^{\s}_d$ be the elements of $\cF^{\s}$ of degree $d$.
We use $V$ to denote the set of monic polynomials in $\F_q[x]$, and $V_d\sub V$ to denote those of degree $d$.

We define the zeta function
$$
\zeta(s):=\prod_{\substack{P \in \mathbb{F}_{q}[x]: \text{ P is monic irreducible}}}(1 - q^{-s\deg(P)})=\frac{1}{
1-q^{1-s}}
$$
and the M\"{o}bius function for $f\in \F_q[x]$
$$
\mu(f)  
	:=
	\begin{cases}
	0 \ \ \ \ \ \ \text{ if $f$ not not square-free}, \\
	(-1)^{k} \text{ if $f$ is the product of $k$ distinct irreducible factors}.
	\end{cases}
$$

\subsection{Lemmas}
The following is an analog of \cite[Lemma 5]{kurlberg-rudnick} (which is the special case $n=2$).
We count the number of $n$th power-free polynomials interpolating given values.
\begin{lemma}\label{L:count}
  Let $x_{1}, \cdots, x_{l} \in \mathbb{F}_{q}$ be distinct elements, and fix any  $a_{1}, \cdots, a_{l} \in \mathbb{F}_{q}^{*}$. Then
$$\#\{f \in \mathcal{F}_{d} \ | \ f(x_{1}) = a_{1}, \cdots, f(x_{l}) = a_{l}\} = \dfrac{q^{d-l}(q-1)}{\zeta(n)(1-q^{-n})^{l}} + O(q^{d/n+1}).$$
\end{lemma}

\begin{proof}
By inclusion-exclusion, we have

\begin{align*}
 &\#\{f \in \mathcal{F}_{d} \ | \ f(x_{1}) = a_{1}, \cdots, f(x_{l}) = a_{l}\}\\
=&\displaystyle\sum_{f_{2} \in V:0 \leq \deg(f_{2}) \leq d/n} \mu(f_{2})\#\{f_{1} \in \F_q[x]_{d - n\deg(f_{2})} : f_{2}(x_{j})^{n}f_{1}(x_{j}) = a_{j} \text{ for } 1 \leq j \leq l\}\\
=&\displaystyle\sum_{f_{2} \in V:0 \leq \deg(f_{2}) \leq d/n} \mu(f_{2})\#\{f_{1} \in \F_q[x]_{d - n\deg(f_{2})} : f_{1}(x_{j}) = a_{j}/f_{2}(x_{j})^{n} \text{ for } 1 \leq j \leq l\}.
\end{align*}

If $D\geq l$, there are $(q-1)q^{D-\ell}$ degree $D$ elements of $\F_q[x]$
that interpolate $l$ given values because each interpolation point imposes a linearly independent condition on the 
coefficients of a monic polynomial of degree $D$ (by the Vandermonde determinant).  If $l\geq D+1$, then $D+1$ of these conditions are still linearly independent on polynomials of degree $D$, and
thus there is at most $1$ polynomial of degree $D$ interpolating $l$ given values.

We split the above expression into two sums depending on the degree of $f_2$, and thus the above is  
\begin{align*}
= & \displaystyle\sum_{\substack{f_{2} \in V:0 \leq \deg(f_{2}) \leq (d-l)/n \\ f_{2}(x_{j}) \neq 0 \ \text{ for } 1 \leq j \leq l}} \mu(f_{2}) (q-1) q^{d-n\deg(f_{2})-l} + O\left(\displaystyle\sum_{\substack{f_{2}\in V:(d-l)/n < \deg(f_{2}) \leq d/n \\ f_{2}(x_{j}) \neq 0 \ \text{ for } 1 \leq j \leq l}} 1\right).
\end{align*}

We bound the error term by counting all polynomials of degree at most $d/n$ to see that the above is
\begin{align*}
= & \displaystyle\sum_{\substack{f_{2} \in V:0 \leq \deg(f_{2}) \leq (d-l)/n \\ f_{2}(x_{j}) \neq 0 \ \text{ for } 1 \leq j \leq l}} \mu(f_{2}) (q-1) q^{d-n\deg(f_{2})-l} + O(q^{d/n})\\
=&q^{d-l}\displaystyle\sum_{\substack{f_{2} \in V \\ f_{2}(x_{j}) \neq 0 \ \text{ for } 1 \leq j \leq l}} \mu(f_{2}) (q-1) q^{-n\deg(f_{2})} - q^{d-l}\displaystyle\sum_{\substack{f_{2} \in V:(d-l)/n < \deg(f_{2})  \\ f_{2}(x_{j}) \neq 0 \ \text{ for } 1 \leq j \leq l}} \mu(f_{2}) (q-1)q^{-n\deg(f_{2})} + O(q^{d/n}).
\end{align*}

We bound the middle term by counting that there are at most $q^i$ monic polynomials of degree $i$, and summing the geometric series to obtain that the above is
\begin{align*}
=&q^{d-l}\displaystyle\sum_{\substack{f_{2} \in V \\ f_{2}(x_{j}) \neq 0 \ \text{ for } 1 \leq j \leq l}} \mu(f_{2}) (q-1) q^{-n\deg(f_{2})}  + O(q^{d/n+1}).
\end{align*}

We have 
\begin{align*}
 \displaystyle\sum_{\substack{f_{2} \in V \\ f_{2}(x_{j}) \neq 0 \ \text{ for } 1 \leq j \leq l}} \mu(f_{2}) (q-1) q^{-n\deg(f_{2})}=&(q-1) \prod_{\substack{P \in \mathbb{F}_{q}[x]: \text{ P is monic irreducible} \\ P(x_{j}) \neq 0 \ \text{ for } 1 \leq j \leq l}}(1 - q^{-n\deg(P)})\\
=&(q-1) \zeta(n)^{-1} (1-q^{-n})^{-l}.
\end{align*}

Thus
\begin{align*}
 &\#\{f \in \mathcal{F}_{d} \ | \ f(x_{1}) = a_{1}, \cdots, f(x_{l}) = a_{l}\}
=\frac{q^{d-l}
(q-1)}{ \zeta(n) (1-q^{-n})^{l}}
 + O(q^{d/n+1}).
\end{align*}

\end{proof}

Now we refine our interpolation count to $\cF^{\s}_d.$

\begin{lemma}\label{L:counts} Let $a_1,\dots,a_q\in\F_q^*$.  Then
\begin{align*}
 &\# \{f\in \cF^{\s}_d | \textrm{ $\frac{f}{(x-x_i)^{s_i}}|_{x=x_i}=a_i$ for all $i$}\}=
\frac{q^{d-\sum{s_i}-q}(q-1)}{\zeta(n)(1-q^{-n})^{q}} + O(q^{(d-\sum{s_i})/n+1}).
\end{align*}
\end{lemma}

\begin{proof}
 For some $b_i\in \F_q^*$, we have
\begin{align*}
 \# \{f\in \cF^{\s}_d | \textrm{ $\frac{f}{(x-x_i)^{s_i}}|_{x=x_i}=a_i$ for all $i$}\}&=
\# \{f\in \cF_{d-\sum{s_i}} | \textrm{ $f(x_i)=b_i$ for all $i$}\}
\end{align*}
and then we apply Lemma~\ref{L:count}.
\end{proof}

\begin{lemma}\label{L:dem}
We have that
$\# \cF_d=(q-1)(q^d-q^{d-n+1})=q^d(q-1)/\zeta(n)$ for $d\geq n.$
\end{lemma}

\begin{proof}
 By counting all monic polynomials of degree $d$, we have
$$\frac{1}{1-qt}=\sum_{d\geq 0} (qt)^d =\prod_{\substack{P \in \mathbb{F}_{q}[x]: \text{ P is monic irreducible}}}
\frac{1}{1-t^{\deg(P)}}.$$
By counting $n$th power-free monic polynomials, we then have
\begin{align*}
 \sum_d \frac{\# \cF_d}{q-1} t^d=
&\prod_{\substack{P \in \mathbb{F}_{q}[x]: \text{ P is monic irreducible}}}(1 + t^{\deg(P)}
+\dots + t^{(n-1)\deg(P)})\\
=&\prod_{\substack{P \in \mathbb{F}_{q}[x]: \text{ P is monic irreducible}}}
\frac{1-t^{n\deg(P)}}{1-t^{\deg(P)}}\\=&\frac{1-qt^n}{1-qt}.
\end{align*}
We conclude that $\# \cF_d=(q-1)(q^d-q^{d-n+1})=q^d(q-1)/\zeta(n)$ for $d\geq n.$
\end{proof}

See also \cite[Proposition 5.9(a)]{vakil-wood}, which gives a much more general proof of this kind of identity (with the line
replaced by any variety), and without using the Euler product.

\subsection{Proof of Theorem~\ref{T:nfree}}

We are now ready to prove our main theorem, which we do by reducing it to the interpolation counts of
Lemma~\ref{L:counts}.
\begin{proof}[Proof of Theorem~\ref{T:nfree}]

Let $\phi(s_i,k)$ (respectively 
$\tilde{\phi}(s_i,k)$) be the number 
 of $a_i\in \F_q^*$ such that for any $f\in \cF$ with $(x-x_i)^{s_i}||f$, we have that $\frac{f}{(x-x_i)^{s_i}}|_{x=x_i}=a_i$ implies
that $C_f$ (respectively $\tilde{C_f}$) has $k$ points with $x=x_i$.

Applying Lemma~\ref{L:counts}, we have
\begin{align*}
 &\# \{f\in \cF_d | \textrm{ $C_f$ has $k_i$ points with $x=x_i$ for all $i$}\}\\
=& \sum_{0\leq s_1,\dots, s_q \leq n-1} 
\# \{f\in \cF^{\s}_d | \textrm{ $C_f$ has $k_i$ points with $x=x_i$ for all $i$}\}\\
=& \sum_{0\leq s_1,\dots, s_q \leq n-1}  \prod_{i=1}^q {\phi}(s_i,k_i)
\# \{f\in \cF^{\s}_d | \textrm{ $\frac{f}{(x-x_i)^{s_i}}|_{x=x_i}=a_i$ for all $i$}\}\\
=& \sum_{0\leq s_1,\dots, s_q \leq n-1}  \prod_{i=1}^q {\phi}(s_i,k_i)
\left( \frac{q^{d-\sum{s_i}-q}(q-1)}{\zeta(n)(1-q^{-n})^{q}} + O(q^{(d-\sum{s_i})/n+1}) \right)\\
=& \frac{q^{d-q}(q-1)}{\zeta(n)(1-q^{-n})^{q}}
\sum_{0\leq s_1\leq n-1}  {\phi}(s_1,k_1)q^{-s_1} \cdots \sum_{0\leq s_q\leq n-1}  {\phi}(s_q,k_q)q^{-s_q}
+ O(n^qq^{q+d/n+1}),
\end{align*}
and similarly for $\tilde{C_f}$ and $\tilde{\phi}$.

We now determine $\phi(s_i,k)$ and $\tilde{\phi}(s_i,k)$.
If a degree $d$ polynomial $f$ in $\F_q[x]$ is an $r$th power in $\bar{\F_q}[x]$ for some $r>1$, then $f=cg^r$
for some $c\in \F_q$ and monic $g\in \F_q[x]$, and so there are at most $q^{d/r+1}$ of these polynomials $f$.
Using  Lemma~\ref{L:dem}, we see that in the limit as $d\ra\infty$, these
 polynomials will not contribute to the total proportion.
By Lemma~\ref{L:irred}, this gives that the geometrically reducible $C_f$ do not contribute in 
the $d\ra\infty$ limit, giving the second statement of the theorem.
We now  consider only $f$ for which the power of $f$ over $\bar{\F_q}$ is relatively prime to $m$.

For $C_f$, if $s_i=0$, then $\frac{q-1}{(m,q-1)}$ of the $q-1$ possible values
 of $a_i$ (more specifically, the $m$th powers) give $k=(m,q-1)$ 
(the number of $m$th roots of an $m$th power in $\F_q^*$, i.e. choices for $y$ given $x=x_i$).
The other
$(q-1)(1-\frac{1}{(m,q-1)})$ values of  $a_i$ give $k=0.$ 
For $C_f$, if $s_i\geq 1$, then all values of $a_i$ give $k=1$ (as necessarily $y=0$).

For $\tilde{C_f}$, then $\frac{q-1}{(m,s_i,q-1)}$ of the $q-1$ possible values of $a_i$ give $k=(m,s_i,q-1)$ and the other
$(q-1)(1-\frac{1}{(m,s_i,q-1)})$ values of  $a_i$ give $k=0$ by Lemma~\ref{L:howmany}.

The first statement of the theorem follows, using Lemma~\ref{L:dem}.

\end{proof}

\subsection*{Acknowledgements} 

We would like to thank Alina Bucur for useful conversations. 
 Theorem~\ref{T:sqfree} was given as a conjecture by E. Holzhausen and M. Willer as a result of work supported by
NSF grant
 DMS-0838210. 
M. M. Wood was supported in this work by an American Institute of Mathematics Five-Year Fellowship and NSF grant DMS-1147782.  
G. Cheong and A. Zaman were supported by NSF grants DMS-1147782
 and 
 DMS-0838210.

\bibliographystyle{alpha}
\bibliography{myrefs.bib}

\def\cprime{$'$}
\begin{thebibliography}{BDFL10b}

\bibitem[BDFL10a]{bdfl-planar}
Alina Bucur, Chantal David, Brooke Feigon, and Matilde Lal{\'{\i}}n.
\newblock Fluctuations in the number of points on smooth plane curves over
  finite fields.
\newblock {\em J. Number Theory}, 130(11):2528--2541, 2010.

\bibitem[BDFL10b]{bdfl-trigonal}
Alina Bucur, Chantal David, Brooke Feigon, and Matilde Lal{\'{\i}}n.
\newblock Statistics for traces of cyclic trigonal curves over finite fields.
\newblock {\em Int. Math. Res. Not. IMRN}, (5):932--967, 2010.

\bibitem[BDFL11]{bdfl-pfold}
Alina Bucur, Chantal David, Brooke Feigon, and Matilde Lal{\'{\i}}n.
\newblock Biased statistics for traces of cyclic {$p$}-fold covers over finite
  fields.
\newblock In {\em W{IN}---women in numbers}, volume~60 of {\em Fields Inst.
  Commun.}, pages 121--143. Amer. Math. Soc., Providence, RI, 2011.

\bibitem[BK11]{bucur-kedlaya}
Alina Bucur and Kiran Kedlaya.
\newblock The probability that a complete intersection is smooth.
\newblock 2011.
\newblock to appear in Journal de Théorie des Nombres de Bordeaux,
  arXiv:1003.5222.

\bibitem[EW12]{erman-wood-semiample}
Daniel Erman and Melanie~Matchett Wood.
\newblock Semiample bertini theorems over finite fields.
\newblock 2012.
\newblock arXiv:1209.5266.

\bibitem[KR09]{kurlberg-rudnick}
P{\"a}r Kurlberg and Ze{\'e}v Rudnick.
\newblock The fluctuations in the number of points on a hyperelliptic curve
  over a finite field.
\newblock {\em J. Number Theory}, 129(3):580--587, 2009.

\bibitem[KS99]{katz-sarnak-a}
Nicholas~M. Katz and Peter Sarnak.
\newblock {\em Random matrices, {F}robenius eigenvalues, and monodromy},
  volume~45 of {\em American Mathematical Society Colloquium Publications}.
\newblock American Mathematical Society, Providence, RI, 1999.

\bibitem[VW12]{vakil-wood}
Ravi Vakil and Melanie~Matchett Wood.
\newblock Discriminants in the {Grothendieck} ring.
\newblock 2012.
\newblock arXiv:1208.3166.

\bibitem[Woo12]{wood-trigonal}
Melanie~Matchett Wood.
\newblock The distribution of the number of points on trigonal curves over
  $\mathbb{F}_q$.
\newblock {\em International Mathematics Research Notices}, 2012.

\end{thebibliography}

\end{document}